\pdfoutput=1
\documentclass[12pt, a4paper]{amsart}

\numberwithin{equation}{section}

\usepackage{amsmath,amsfonts,amssymb}
\usepackage{mathtools}
\usepackage[utf8]{inputenc}
\usepackage[english]{babel}
\usepackage{hyperref}
\usepackage{tikz-cd}

\usepackage[margin=1.2in,footskip=0.25in]{geometry}

\DeclareUnicodeCharacter{FB01}{fi}
\newtheorem{theorem}{Theorem}[section]

\newtheorem{lemma}[theorem]{Lemma}
\newtheorem{prop}[theorem]{Proposition}

\theoremstyle{definition}
\newtheorem{definition}{Definition}[section]

\begin{document}
	\title[ON BIRATIONAL TORELLI THEOREMS]
	{On birational Torelli theorems}
	
	\author{Sumit Roy}
	\address{School of Mathematics, Tata Institute of Fundamental Research, Homi Bhabha Road, Colaba, Mumbai 400005, India.}
	\email{sumit@math.tifr.res.in}
	\thanks{E-mail : sumit@math.tifr.res.in}
	\thanks{Address : School of Mathematics, Tata Institute of Fundamental Research, Homi Bhabha Road, Colaba, Mumbai 400005, India.}
	
	\begin{abstract}
		Let $G$ be a simple simply-connected connected linear algebraic group over $\mathbb{C}$. We proved a $2$-birational Torelli theorem for the moduli space of semistable principal $G$-bundles over a smooth curve of genus $\geq 3$, which says that if two such moduli spaces are $2$-birational then the curves are isomorphic. We also proved a $3$-birational Torelli theorem for the moduli space of stable symplectic parabolic bundles over a smooth curve of genus $\geq 4$.
	\end{abstract}
	
	\maketitle

	\section{Introduction}
	Let $X$ and $X'$ be two compact Riemann surfaces. The classical Torelli theorem says that if their Jacobian varieties $\mathrm{Jac}(X)$ and $\mathrm{Jac}(X')$ are isomorphic as polarized varieties, with the canonical polarization given by the theta line bundle, then $X$ is isomorphic to $X'$.
	
	There are several nonabelian analogues of the Torelli theorem, by considering the moduli spaces of vector bundles instead of Jacobians. In \cite{MN68}, Mumford and Newstead proved a Torelli theorem for the moduli space of vector bundles with rank $2$ and fixed determinant of odd degree. Later, in \cite{NR75}, Narasimhan and Ramanan extended this result for any rank. 
	
	In \cite{BH12}, Biswas and Hoffmann proved a Torelli theorem for the moduli space of semistable principal $G$-bundles ($G$ is a connected reductive complex affine algebraic group) over $X$ of genus $\geq 3$ of topological type $d \in \pi_1(G)$. They considered the strictly semistable locus and showed that it lies in the singular locus and is characterized by the types of singularities. Then they considered a morphism to a projective space using the powers of its anticanonical line bundle, and using the fiber of this morphism they were able to reconstruct the Jacobian of $X$ and its canonical principal polarization. Hence the result follows from the classical Torelli theorem. 
	
	In \cite{BBB01}, Balaji, Baño and Biswas proved a Torelli theorem for parabolic bundles of rank $2$, fixed determinant of degree $1$ and small parabolic weights (assuming $g\geq 2$). Later in \cite{AG19}, Alfaya and Gómez extended this result for parabolic bundles of any rank assuming $g \geq 4$ and fixed determinant. In \cite{R21}, a Torelli theorem was proved for the moduli space of stable symplectic parabolic bundles over a compact Riemann surface of genus $\geq 4$, with rank and degree coprime, small parabolic weights.
	
 In this paper we address two birational Torelli theorems, one for the moduli space of semistable principal bundles and another for the moduli space of stable symplectic parabolic bundles. Two varieties $V$ and $V'$ are called $k$-birational if there exist two open sets $U \subset V$ and $U' \subset V'$ whose respective complements have codimension $\geq k$ and an isomorphism $\phi : U \xrightarrow{\sim} U'$. 
 
 Let $G$ be a simple simply-connected connected affine algebraic group over $\mathbb{C}$. We will prove that if two moduli spaces of semistable principal $G$-bundles over compact Riemann surfaces $X$ and $X'$ (genus $\geq 3$) are $2$-birational, then $X$ is isomorphic to $X'$ [Theorem \ref{thm1}].  In \cite{AB21}, Alfaya and Biswas proved a similar result for the moduli space of vector bundles.
 
 Let $X$ and $X'$ be two compact Riemann surfaces of genus $g \geq 4$ and $g' \geq 4$ respectively with the set of marked points $D\subset X$ and $D' \subset X'$. Let $L$ and $L'$ be two parabolic line bundles with trivial parabolic structure. Let $\mathcal{M}_{\mathrm{Sp}}(2m,\alpha,L)$ and $\mathcal{M}_{\mathrm{Sp}}(2m',\alpha',L')$ be two moduli spaces of stable symplectic parabolic bundles over $X$ and $X'$ respectively with rank and degree coprime, small parabolic weights. We will show that if $\mathcal{M}_{\mathrm{Sp}}(2m,\alpha,L)$ and $\mathcal{M}_{\mathrm{Sp}}(2m',\alpha',L')$ are $3$-birational then there exists an isomorphism $X \cong X'$ sending $D$ to $D'$ [Theorem \ref{thm2}]. Alfaya and Gomez in their paper \cite{AG19} proved a $3$-birational Torelli theorem for the moduli space of parabolic bundles.

\section{Preliminaries}
Let $X$ be a compact Riemann surface. Let $G$ be a simple simply-connected connected complex affine algebraic group and let $\mathfrak{g}=\mathrm{Lie}(G)$ be the Lie algebra of $G$.

\subsection{Principal $G$-bundles}
We consider a holomorphic principal $G$-bundle $\pi : P \to X$ over $X$. Let $\mathrm{Ad} : G \to \operatorname{End}(\mathfrak{g})$ be the adjoint action of $G$ on $\mathfrak{g}$. This induces an action on $P \times \mathfrak{g}$ and the associated bundle $P(\mathrm{ad}) \coloneqq P\times^{\mathrm{Ad}} \mathfrak{g}$ is called the \textit{adjoint bundle} of $P$, denoted by $\mathrm{ad}(P)$.

A principal $G$-bundle $P$ is called \textit{stable} (resp. \textit{semistable}) if for every holomorphic reduction of structure group $P_H$ of $P$ to a maximal parabolic subgroup $H \subsetneqq G$, we have
\[
\deg (\mathrm{ad}(P_H)) < 0 \hspace{0.3cm} (\mathrm{resp. } \hspace{0.2cm} \leq \hspace{0.05cm}).
\]

 The moduli space of semistable principal $G$-bundles over $X$ was constructed by Ramanathan in \cite{R96}. It is a normal irreducible projective variety of dimension $(g-1)\dim G$, where $g$ is the genus of $X$.

	\subsection{Parabolic vector bundles}Let $D \subset X$ be a subset of $n$ distinct points of $X$. A \textit{parabolic vector bundle} $E_*$ of rank $r$ on $X$ is a holomorphic vector bundle $E$ of rank $r$ on $X$ together with a parabolic structure along $D$, i.e. for each $p \in D$, we have
\begin{enumerate}
	\item a filtration of subspaces  $$\left.E\right|_p= E_{p,1}\supsetneq \dots \supsetneq E_{p,r_p} \supsetneq E_{p,r_p+1} = \{0\}, $$
	\item a sequence of real numbers (weights)  satisfying $$0\leq \alpha_1(p) < \alpha_2(p) < \dots < \alpha_{r_p}(p) < 1.$$
\end{enumerate}
The weight $\alpha_i(p)$ corresponds to the subspace $E_{p,i}$. We fix the finite subset $D$ once and for all.  

We denote $\alpha = \{(\alpha_1(p),\dots,\alpha_{r_p}(p)) \}_{p \in D}$ to the system of weights corresponding to the fixed parabolic structure. The parabolic structure is said to have \textit{full flags} whenever $\dim(E_{p,i}/E_{p,i+1}) = 1$ \hspace{0.2cm}$\forall i, \forall p\in D$. 

The \textit{parabolic degree} of a parabolic bundle $E_*$ is defined by
\[
\operatorname{pardeg}(E_*) \coloneqq \deg(E)+ \sum\limits_{p\in D}\sum\limits_{i=1}^{r_p}\alpha_i(p) \cdot \dim(E_{p,i}/E_{p,i+1})
\]
and the real number $$\mu_{par}(E_*)\coloneqq \frac{\operatorname{pardeg}(E_*)}{\mathrm{rk}(E)}$$ is called the \textit{parabolic slope} of $E_*$. The dual and tensor product of parabolic bundles can be defined in a natural way (see \cite{Y95}).

A \textit{parabolic homomorphism} $\phi : E_* \to E^\prime_*$ between two parabolic vector bundles is a vector bundle homomorphism which satisfies
the following: at each $p \in D$ we have 
\[
\phi_p(E_{p,i}) \subseteq E_{p,j+1}^\prime \hspace{0.2cm} \mathrm{whenever} \hspace{0.2cm} \alpha_i(p) > \alpha_{j}^\prime(p).
\] 
Furthermore, we call such morphism \textit{strongly parabolic} if 
\[
\phi_p(E_{p,i}) \subseteq E_{p,j+1}^\prime \hspace{0.2cm} \mathrm{whenever} \hspace{0.2cm} \alpha_i(p) \geq \alpha_{j}^\prime(p).
\]
for every $p \in D$. We denote the parabolic endomorphisms of $E_*$ by $\operatorname{PEnd}(E_*)$ and strongly parabolic endomorphisms by $\operatorname{SPEnd}(E_*)$.

\subsection{Symplectic parabolic Higgs bundles} Fix a parabolic line bundle $L_*$ over $X$. Let $E_*$ be a parabolic bundle over $X$ and let
\begin{equation}\label{form}
\psi : E_* \otimes E_* \to L_*
\end{equation}
be a parabolic homomorphism. Consider the morphism
\[
\psi \otimes Id : (E_* \otimes E_*) \otimes E^\vee_* \to L_* \otimes E^\vee_*,
\]
where $E_*^\vee$ denote the parabolic dual of $E_*$. A section of the vector bundle underlying $E_* \otimes E_*^\vee$ is an endomorphism of $E$ preserving the quasi-parabolic structure. The trivial line bundle $\mathcal{O}_X$ equipped with the trivial parabolic structure (meaning all weights are zero) can be realized as a parabolic subbundle of $E_* \otimes E^\vee_*$ by sending a locally defined function $f$ to the locally defined endomorphism of $E$ given by pointwise multiplication with $f$. Let
\[
\tilde{\psi} : E_* \to L_* \otimes E^\vee_*
\]
be the homomorphism defined by the composition
\[
E_* = E_* \otimes \mathcal{O}_X  \xhookrightarrow{} E_* \otimes (E_* \otimes E^\vee_*) = (E_* \otimes E_*) \otimes E^\vee_* \xrightarrow{\psi \otimes Id} L_* \otimes E^\vee_*.
\]
\begin{definition} A \textit{symplectic parabolic bundle} is a pair $(E_*,\psi)$ of the above form such that $\psi$ is anti-symmetric and $\tilde{\psi}$ is an isomorphism.
\end{definition}

For a symplectic parabolic bundle $(E_*,\psi)$, we denote the symplectic parabolic endomorphisms by $\mathrm{PEnd}_{\mathrm{Sp}}(E_*) \subset \mathrm{PEnd}(E_*)=E_* \otimes E_* \otimes L_*^\vee$.

\begin{definition} Let $K$ denote the canonical bundle on $X$. We write $K(D) \coloneqq K \otimes \mathcal{O}(D)$. A \textit{strongly parabolic Higgs bundle} on $X$ is a pair $(E_*,\Phi)$ where $E_*$ is a parabolic bundle on $X$ and $\Phi$ is a morphism $\Phi : E_* \to E_* \otimes K(D)$ which is strongly parabolic. The morphism $\Phi$ is called the \textit{strongly parabolic Higgs field} associated to the bundle $E_*$.
\end{definition}

Throughout this paper, we will always assume that the parabolic Higgs field is strongly parabolic.

Let $(E_*,\psi)$ be a symplectic parabolic bundle over $X$. A parabolic Higgs field on $E_*$ will induce a parabolic Higgs field on $L_* \otimes E^\vee_*$ (here we are considering the zero section as the Higgs field on $L_*$). A parabolic Higgs field $\Phi$ on $(E_*,\psi)$ is said to be compatible with  $\psi$ if $\tilde{\psi}$ takes $\Phi$ to the induced parabolic Higgs field on $L_* \otimes E^\vee_*$.
\begin{definition}
A \textit{symplectic parabolic Higgs bundle} $(E_*,\psi,\Phi)$ is a symplectic parabolic bundle $(E_*,\psi)$  together with a parabolic Higgs field $\Phi$ on $E_*$ which is compatible with $\psi$.
\end{definition}

Suppose $E$ is the underlying vector bundle of a symplectic parabolic bundle $(E_*,\psi)$. The vector bundle tensor product $E \otimes E$ is a coherent subsheaf of the vector bundle underlying the parabolic vector bundle $E_*\otimes E_*$. Therefore, $\psi$ induces a morphism
\begin{equation*}
	\hat{\psi} : E \otimes E \to L,
\end{equation*}
where $L$ is the line bundle underlying $L_*$. A holomorphic subbundle $F \subset E$ is called \textit{isotropic} if $\hat{\psi}(F \otimes F) = 0$. The parabolic structure on $E$ induces a parabolic structure on the subbundle $F$. Let $F_*$ be the parabolic bundle with the induced parabolic structure on $F$.

\begin{definition}
	A symplectic parabolic bundle $(E_*,\psi)$ is called \textit{stable} (resp. \textit{semistable}) if every nonzero isotropic subbundle $F \subset E$ satisfies
	\[
	\mu_{par}(F_*) < \mu_{par}(E_*) \hspace{0.3cm}(\mathrm{resp.} \hspace{0.2cm} \leq \hspace{0.05cm} ).
	\]
\end{definition}

\begin{definition} A symplectic parabolic Higgs bundle $(E_*,\psi,\Phi)$ is called \textit{stable} (resp. \textit{semistable}) if every nonzero isotropic subbundle $F \subset E$ such that $\Phi(F) \subset F \otimes K(D)$ satisfies
	\[
	\mu_{par}(F_*) < \mu_{par}(E_*) \hspace{0.3cm}(\mathrm{resp.} \hspace{0.2cm} \leq \hspace{0.05cm} ).
	\]
\end{definition}

When all weights are rational, the notion of symplectic parabolic bundle coincides with the notion of parabolic principal $G$-bundle where $G$ is a complex symplectic group (see \cite{BBN01, BBN03, BMW11} for details).

The moduli space of semistable parabolic $G$-bundles with a fixed parabolic structure $\alpha$ was described in \cite{BR89} and \cite{BBN01}. It is a complete normal variety and the moduli space of stable parabolic $G$-bundles is an open subvariety. Fix a parabolic line bundle $L$ with the trivial parabolic structure.  
Let $\mathcal{M}_{\text{Sp}}(2m,\alpha,L)$ denote the moduli space of stable symplectic parabolic bundles of rank $2m$ $ (m > 1)$,  and fixed parabolic structure $\alpha$, with the symplectic form (\ref{form}) taking values in $L$. When the parabolic structure $\alpha$ has full flags
\[
\dim \mathcal{M}_{\text{Sp}}(2m,\alpha,L) = m(2m+1)(g-1) + m^2n,
\]
where $n$ is the number of marked points (i.e. the number of points in $D$) on $X$ (see \cite[Theorem II]{BR89}). From now on, we assume that the weights are all rational and the parabolic structure has full flags at each point in $D$.

Let $\mathcal{N}_{\text{Sp}}(2m,\alpha,L)$ denote the moduli space of stable symplectic parabolic Higgs bundles of rank $2m$ (see \cite{R16}). By considering the zero Higgs fields, we have an embedding $\mathcal{M}_{\text{Sp}}(2m,\alpha,L) \xhookrightarrow[]{} \mathcal{N}_{\text{Sp}}(2m,\alpha,L)$. By the parabolic Serre duality (see \cite[p. $1470$]{BMW11}, \cite{Y95,BY96}), the cotangent bundle $T^*\mathcal{M}_{\text{Sp}}(2m,\alpha,L)$ is contained inside the moduli space $\mathcal{N}_{\text{Sp}}(2m,\alpha,L)$ as an open subset. Therefore,
\[
\dim\mathcal{N}_{\text{Sp}}(2m,\alpha,L) = 2\dim \mathcal{M}_{\text{Sp}}(2m,\alpha,L) = 2m(2m+1)(g-1) + 2m^2n.
\]

\begin{definition}
	Let $k$ and $l$ be two integers. A symplectic parabolic bundle $(E_*,\psi)$ is $(k,l)$-\textit{stable} (resp. $(k,l)$-\textit{semistable}) if every nonzero isotropic subbundle $F \subset E$ satisfies 
	\[
	\frac{\text{pardeg}(F_*) + k}{\mathrm{rk}(F)} < \frac{\text{pardeg}(E_*) - l}{\text{rk}(E)} \hspace{0.5cm}(\mathrm{resp.  } \hspace{0.2cm} \leq \hspace{0.05cm}).
	\]	
\end{definition}
Observe that if $k$ and $l$ are nonnegative, then a $(k,l)$-stable symplectic parabolic bundle is stable in the usual sense.

\begin{prop}\cite[Proposition $1$]{R21}\label{prop1}
	For $ g \geq 3$, the locus of $(1,0)$-stable symplectic parabolic bundles is a non-empty Zariski open subset of $\mathcal{M}_{\textnormal{Sp}}(2m,\alpha,L)$.
\end{prop}

 \begin{lemma}\cite[Lemma $2.3$]{R21}\label{lemma1}
	Suppose $g \geq 4$ and the weights are small enough so that the stability of the symplectic parabolic Higgs bundle is equivalent to the stability of the underlying vector bundle. Then $H^0(\textnormal{PEnd}_{\mathrm{Sp}}(E_*)(x))=0$ for a generic stable symplectic parabolic bundle $E_*$.
\end{lemma}	

The following proposition is exactly \cite[Proposition $2$]{R21} and the proof can be found in \cite[pp.$10$-$11$]{BGM12}.
\begin{prop}\label{3}
	Let $Y$ be an integral curve which has a unique simple node. Also assume that $Y$ possesses an involution $\sigma$ and let $\pi_Y : \tilde{Y} \to Y$ be the normalization. Then the compactified Jacobian $\bar{J}(Y)$ is birational to a $\mathbb{P}^1$-fibration over $J(\tilde{Y})$. \\
	Analogously, let $Y$ be an integral curve with two simple nodes which possesses an involution $\sigma$ which interchanges these two nodes, and let $\tilde{Y}$ be the normalization of $Y$. Then $\bar{J}(Y)$ is birational to a $\mathbb{P}^1 \times \mathbb{P}^1$-bundle on $J(\tilde{Y})$.\\
	And in either case the Prym variety, which is the fixed point variety of the involution, is an uniruled variety.
\end{prop}

	\section{Birational geometry}
	
	\begin{definition}
		Let  $V$ and $V'$ be two varieties and let $k$ be a positive integer. A \textit{$k$-birational} morphism between $V$ and $V'$ is an isomorphism $\phi : U \xrightarrow{\sim} U'$ between two open subsets $U\subset V$ and $U'\subset V'$ such that
		\begin{align*}
			\mathrm{codim}(V\setminus U) &\geq k\\
			\mathrm{codim}(V'\setminus U') &\geq k.
		\end{align*}	
	We say that the varieties $V$ and $V'$ are $k$-birational if there exists a $k$-birational morphism between them.
	\end{definition}
Observe that when $k=1$, a $1$-birational map is exactly the same as a birational map, i.e. two varieties $V$ and $V'$ are birational if they are at least $1$-birational. There are some $k$-birational invariants of a variety which are not invariants under a birational map. For example, if two normal connected varieties $V$ and $V'$ are $2$-birational then by Hartog's theorem we have an isomorphism $H^0(V,\mathcal{O}_V) \cong H^0(V',\mathcal{O}_{V'})$, but this is not true when they are only birational.
\begin{prop}\label{pic}
	Let $\mathcal{M}$ and $\mathcal{M}'$ be two normal quasi-projective varieties. If $\mathcal{M}$ and $\mathcal{M}'$ are $2$-birational then $\mathrm{Pic}(\mathcal{M}) \cong \mathrm{Pic}(\mathcal{M}')$ and $\pi_1(\mathcal{M}) \cong \pi_1(\mathcal{M}')$. 
\end{prop}
\begin{proof}
Since $\mathcal{M}$ and $\mathcal{M}'$ are $2$-birational, there exist open subsets $\mathcal{U}\subset \mathcal{M}$ and $\mathcal{U}' \subset \mathcal{M}'$ whose complements have codimension at least $2$ and an isomorphism $\phi : \mathcal{U} \to \mathcal{U}'$. Since $\mathcal{M}$ is a normal variety, so is the open subset $\mathcal{U}$. Also, every line bundle over $\mathcal{U}$ extends to a line bundle over $\mathcal{M}$ as $\mathrm{codim}(\mathcal{M}\setminus \mathcal{U}) \geq 2$. Therefore, $\mathrm{Pic}(\mathcal{U})\cong \mathrm{Pic}(\mathcal{M})$. Similarly, $\mathrm{Pic}(\mathcal{U}')\cong \mathrm{Pic}(\mathcal{M}')$. Hence the isomorphism $\phi$ induces
\[
\mathrm{Pic}(\mathcal{M}) \cong \mathrm{Pic}(\mathcal{U}) \cong \mathrm{Pic}(\mathcal{U}')\cong \mathrm{Pic}(\mathcal{M}').
\]
From \cite[p. $42$]{M80}, it follows that $\pi_1(\mathcal{U}) \cong \pi_1(\mathcal{M})$ and $\pi_1(\mathcal{U}') \cong \pi_1(\mathcal{M}')$. Therefore, $\phi$ induces 
\[
\pi_1(\mathcal{M}) \cong \pi_1(\mathcal{U}) \cong \pi_1(\mathcal{U}') \cong \pi_1(\mathcal{M}').
\]
\end{proof}	

In particular, if $\mathcal{M}_{\text{Sp}}(2m,\alpha,L)$ and $\mathcal{M}_{\text{Sp}}(2m',\alpha',L')$ are $2$-birational then their Picard group and the fundamental groups are same.

\begin{prop}\cite[Proposition $2.2$]{AB21}\label{keyprop}
	Let $\mathcal{M}$ and $\mathcal{M}'$ be two normal projective varieties and let $\mathrm{Pic}(\mathcal{M})\cong \mathbb{Z}$. If $\mathcal{M}$ and $\mathcal{M}'$ are $2$-birational then $\mathcal{M} \cong \mathcal{M}'$.
\end{prop}
\begin{proof}
	Let $\mathcal{U}\subset \mathcal{M}$ and $\mathcal{U}' \subset \mathcal{M}'$ be two open subsets such that their respective complements have codimension $\geq 2$ and let $\phi : \mathcal{U} \to \mathcal{U}'$ be an isomorphism. By Proposition \ref{pic},
	\[
	\mathrm{Pic}(\mathcal{M}) \cong \mathrm{Pic}(\mathcal{U}) \cong \mathrm{Pic}(\mathcal{U}')\cong \mathrm{Pic}(\mathcal{M}').
	\]
	Consider a very ample line bundle $\mathcal{L}'$ on $\mathcal{M}'$, and let $\mathcal{M}' \xhookrightarrow[]{} \mathbb{P}(H^0(\mathcal{M}',\mathcal{L}')^\vee)$ be a closed embedding. Since $\mathrm{Pic}(\mathcal{U}) \cong \mathrm{Pic}(\mathcal{M}) \cong \mathbb{Z}$ and $\phi^*: \mathrm{Pic}(\mathcal{U}) \to \mathrm{Pic}(\mathcal{U}')$ is an isomorphism, the line bundle $\phi^*(\left.\mathcal{L}'\right|_{\mathcal{U}'})$ over $\mathcal{U}$ can be uniquely extended to a very ample line bundle $\mathcal{L}$ over $\mathcal{M}$. Therefore, we obtain a closed embedding $\mathcal{M} \xhookrightarrow[]{} \mathbb{P}(H^0(\mathcal{M},\mathcal{L})^\vee)$. Since the complement of $\mathcal{U}$ in $\mathcal{M}$ has codimension at least $2$ and $\mathcal{M}$ is normal, by Hartog's theorem $H^0(\mathcal{U},\mathcal{L}) = H^0(\mathcal{M},\mathcal{L})$. Similarly, $H^0(\mathcal{U}',\mathcal{L}') = H^0(\mathcal{M}',\mathcal{L}')$. Therefore $\phi$ induces an isomorphism
	\[
	\mathbb{P}(H^0(\mathcal{M},\mathcal{L})^\vee) = \mathbb{P}(H^0(\mathcal{U},\mathcal{L})^\vee) \cong \mathbb{P}(H^0(\mathcal{U}',\mathcal{L}')^\vee) = \mathbb{P}(H^0(\mathcal{M}',\mathcal{L}')^\vee)
	\]
	which maps $\mathcal{U}$ to $\mathcal{U}'$. Since the closure of $\mathcal{U}$ (resp. $\mathcal{U}'$) in $\mathbb{P}(H^0(\mathcal{M},\mathcal{L})^\vee)$ (resp. $\mathbb{P}(H^0(\mathcal{M}',\mathcal{L}')^\vee)$) is $\mathcal{M}$ (resp. $\mathcal{M}'$), the morphism $\phi$ extends uniquely to an isomorphism $\mathcal{M} \cong \mathcal{M}'$.
\end{proof}	

\section{Hitchin discriminant}
 We will now discuss the Hitchin map for the moduli space of stable symplectic parabolic Higgs bundles. An element of $\mathcal{N}_{\mathrm{Sp}}(2m,\alpha,L)$ can be viewed as a stable parabolic bundle $E_*$ of rank $2m$ with a non-degenerate symplectic form $\psi$ with values in $L$, together with a morphism $\Phi: E_* \longrightarrow E_* \otimes K(D)$ which satisfies
\[
\psi(\Phi v,w) = - \psi(v,\Phi w).
\]
Let $v_i$, $v_j$ be two eigenvectors of $\Phi$ corresponding to the eigenvalues $\lambda_i$ and $\lambda_j$ respectively. Then
\begin{align*}
	\lambda_i\psi (v_i,v_j) =  \psi(\lambda_iv_i,v_j) =\psi (\Phi v_i,v_j) = -\psi(v_i,\Phi v_j) = -\lambda_j\psi(v_i,v_j).
\end{align*}
Therefore, $\psi(v_i,v_j) = 0$ unless $\lambda_i = - \lambda_j$. Hence, it follows from the nondegeneracy of the symplectic form $\psi$ that if $\lambda_i$ is an eigenvalue of $\Phi$ then so is $-\lambda_i$. Assuming all eigenvalues are distinct, the characteristic polynomial of $\Phi$ has the form
\[
\det(\lambda - \Phi) = \lambda^{2m} + s_2\lambda^{2m-2} + \cdots + s_{2m},
\]
where $s_{2i} = \mathrm{tr}(\wedge^{2i}\Phi) \in H^0(X,K^{2i}(D^{2i}))$ for all $1 \leq i \leq m$. Since $\Phi$ is strongly parabolic, its residue at each point in $D$ is nilpotent and hence $s_{2i} \in H^0(X, K^{2i}(D^{2i-1}))$. Therefore, the Hitchin map is given by
\begin{align*}
	h : \mathcal{N}_{\mathrm{Sp}}(2m,\alpha,L) &\longrightarrow \mathcal{A} \coloneqq \bigoplus_{i=1}^{m} H^0(X, K^{2i}(D^{2i-1}))\\
	(E_*,\varphi,\Phi) &\longmapsto (s_2,s_4,\dots , s_{2m}).
\end{align*}
The dimension of the Hitchin base $\mathcal{A}$ is $m(2m+1)(g-1) + m^2n$, which is the half the dimension of the moduli space $\mathcal{N}_{\mathrm{Sp}}(2m,\alpha,L)$. Also, we consider the restriction map 
\[
h_0 : T^*\mathcal{M}_{\textnormal{Sp}}(2m,\alpha,L) \longrightarrow \mathcal{A}.
\]

Let $\mathcal{S}$ be the total space of the line bundle $K(D)$ and let $p : \mathcal{S} \to X$ be the natural projection. Let $t \in H^0(\mathcal{S},p^*K(D))$ be the tautological section of $p^*K(D)$. Given $s=(s_2,\dots, s_{2m}) \in \mathcal{A}$, the \textit{spectral curve} $X_s$ in $\mathcal{S}$ is defined by
\[
t^{2m} + s_2 t^{2m-2} + \cdots + s_{2m}=0.
\]
Since all exponents of $t$ in the above equation are even, the spectral curve $X_s$ possesses an involution $\sigma(t) = -t$. Therefore, we have a $2$-fold covering map $q: X_s \to X_s/\sigma$. When the spectral curve $X_s$ is smooth, the fiber $h^{-1}(s)$ is identified with the Prym variety $\mathrm{Prym}(X_s, X_s/\sigma)=\{M \in \mathrm{Jac}(X_s) : \sigma^*M \cong M^\vee \}$ \cite[Theorem 4.1]{R20}.

Let $\mathcal{D} \subset \mathcal{A}$ be the divisor consisting of the characteristic polynomials whose corresponding spectral curve is singular. The inverse image $h^{-1}(\mathcal{D})$ is called the \textit{Hitchin discriminant}. If a spectral curve is singular over a point $x \in D$, it is singular precisely at $(x,0)$. Consider the following subsets of $\mathcal{D}$:
\begin{enumerate}
	\item For each parabolic point $x \in D$, let $\mathcal{D}_x$ denote the set of points whose spectral curve is singular over $x$.
    \item Let $\mathcal{D}_1$ denote the set of points whose spectral curve is smooth over every $x\in D$ but singular at some $(y,0)$, where $y \notin D$.
    \item Let $\mathcal{D}_2$ denote the set of points whose spectral curve has two symmetrical nodes (i.e. $t^m + s_2t^{m-1} + \cdots + s_{2m-2}t + s_{2m} = 0$ has a node on a point $(y,t)$ for some $t \ne 0$).
\end{enumerate}
Therefore, $$\mathcal{D} = \underset{x \in D}{\bigcup}\mathcal{D}_x \cup \overline{\mathcal{D}_1} \cup \overline{\mathcal{D}_2},$$ where $\overline{\mathcal{D}_i}$'s are the closure of $\mathcal{D}_i$ in $\mathcal{D}$ for $i=1,2$. 

Since $$\mathcal{D}_x = \bigoplus\limits_{i=1}^{m-1}H^0(K^{2i}D^{2i-1}) \oplus H^0(K^{2m}D^{2m-1}(-x)),$$ it is irreducible for all $x\in D$. 

The set $\overline{\mathcal{D}_1}$ consists of points whose spectral curve is singular at some $(y,0)$ where $y \notin D$ (but not necessarily smooth over $D$). By \cite[Proposition $4.1$]{AG19}, $\overline{\mathcal{D}_1}$ is irreducible. 

By a similar argument as in \cite[Proposition $3.2$]{BGM12}, we can conclude that $\overline{\mathcal{D}_2}$ is irreducible in $\mathcal{D}$.

Generically the singularities in $\mathcal{D}_1$ are nodes and let $\mathcal{D}_1^\circ \subset \mathcal{D}_1$ denote the locus of nodal curves with exactly one node over $y\notin D$. Also, let $\mathcal{D}_2^\circ \subset \mathcal{D}_2$ denote the locus of the curves which do not contain extra singularities.
	
	\begin{prop}\label{4}
		Let $g \geq 4$ and $\mathcal{U} \subset \mathcal{M}_{\mathrm{Sp}}(2m,\alpha,L)$ be an open subset such that the complement has codimension $\geq 3$. Then the complement of $T^*\mathcal{U} \cap h^{-1}(\mathcal{D}_i)$ inside $h^{-1}(\mathcal{D}_i)$ has codimension at least $2$.
	\end{prop}
\begin{proof}
	Let $\mathcal{W}= \mathcal{M}_{\mathrm{Sp}}(2m,\alpha,L) \setminus \mathcal{U}$ and let $N=\dim \mathcal{M}_{\mathrm{Sp}}(2m,\alpha,L)$. Therefore, $\dim \mathcal{W} \leq N-3$. Let $T^*\mathcal{W} \coloneqq \left.T^*\mathcal{M}_{\mathrm{Sp}}(2m,\alpha,L)\right|_{\mathcal{W}}$ denote the restriction of the cotangent of the moduli space to $\mathcal{W}$. Then $\dim (T^*\mathcal{W})\leq 2N-3$, and hence \[
	\dim(T^*\mathcal{W} \cap h^{-1}(\mathcal{D}_i)) \leq 2N-3.
	\]
	For $g\geq 4$, following the computations in Faltings \cite[Theorem II.6(iii)]{F93}, we know that the complement $\mathcal{N}_{\text{Sp}}(2m,\alpha,L) \setminus T^*\mathcal{M}_{\text{Sp}}(2m,\alpha,L)$ has codimension at least $3$. Equivalently,
	\[
	\dim (\mathcal{N}_{\text{Sp}}(2m,\alpha,L) \setminus T^*\mathcal{M}_{\text{Sp}}(2m,\alpha,L)) \leq 2N - 3.
	\]
	Thus, if we denote $\mathcal{Z} = \mathcal{N}_{\text{Sp}}(2m,\alpha,L) \setminus T^*\mathcal{M}_{\text{Sp}}(2m,\alpha,L)$ then
	\[
	\dim (\mathcal{Z} \cap h^{-1}(\mathcal{D}_i)) \leq 2N - 3.
	\]
	Hence,
	\begin{align*}
	\dim (h^{-1}(\mathcal{D}_i)\setminus (T^*\mathcal{U} \cap h^{-1}(\mathcal{D}_i))) &= \dim ((\mathcal{Z}\cap h^{-1}(\mathcal{D}_i))\cup (T^*\mathcal{W} \cap h^{-1}(\mathcal{D}_i)) ) \\
	&\leq 2N - 3 \\
	&= \dim h^{-1}(\mathcal{D}_i) - 2 
	\end{align*}
\end{proof}

\begin{prop}\label{2}
	Let $g \geq 4$ and $\mathcal{U} \subset \mathcal{M}_{\mathrm{Sp}}(2m,\alpha,L)$ be an open subset such that the complement has codimension $\geq 3$. Let $\mathcal{R}_\mathcal{U} \subset T^*\mathcal{U}$ be the union of (complete) rational curves in $T^*\mathcal{U}$. Then $\mathcal{D}$ is the closure of $h(\mathcal{R}_{\mathcal{U}})$ in $\mathcal{A}$.
\end{prop}
\begin{proof}
	The proof is analogous to the proof of \cite[Lemma $8.3$]{AG19}. Let $h_{\mathcal{U}} : T^*\mathcal{U} \longrightarrow \mathcal{A}$ be the restriction of the Hitchin map $h$ to $T^*\mathcal{U}$. Let $l \cong \mathbb{P}^1 \xhookrightarrow{} T^*\mathcal{U}$ be a complete rational curve. So $h_{\mathcal{U}}(l) \subset \mathcal{A}$ is a point as $l$ is a complete curve. Therefore $l$ must be contained in a fiber of the Hitchin map. If $s \in \mathcal{A}\setminus \mathcal{D}$, then the fiber $h^{-1}(s)$ is an abelian variety (a Prym variety). So, $h_{\mathcal{U}}^{-1}(s)$ is an open subset of an abelian variety. Therefore,  $l$ cannnot be contained in a fiber over $\mathcal{A} \setminus \mathcal{D}$. Hence, it is enough to show that for $s\in \mathcal{D}^\circ_i$ and a generic $s\in \mathcal{D}_x$ the fiber $h_{\mathcal{U}}^{-1}(s)$ contains a complete rational curve.
	
	By Proposition \ref{4}, $(\mathcal{M}_{\mathrm{Sp}}(2m,\alpha,L)\setminus T^*\mathcal{U})\cap h^{-1}(\mathcal{D}_i)$ has codimension at least $2$ in $h^{-1}(\mathcal{D}_i)$. So for $s\in \mathcal{D}_i^\circ$,
	\[
	h^{-1}(s) - h_{\mathcal{U}}^{-1}(s) \subset h^{-1}(s)
	\]
	has codimension at least $2$. Therefore by Proposition \ref{3}, $h_{\mathcal{U}}^{-1}(s)$ is an open subset of an uniruled variety. Since the complement of $h_{\mathcal{U}}^{-1}(s)$ has codimension at least $2$, it contains a complete rational curve.
	
	So it remains to show that a generic fiber over $\mathcal{D}_x$ contains a complete rational curve. Let $\mathcal{V} \subset \mathcal{M}_{\textnormal{Sp}}(2m,\alpha,L)$ be the intersection of two open subsets defined by Proposition \ref{prop1} and Lemma \ref{lemma1}, i.e. the elements of $\mathcal{V}$ are $(1,0)$-stable symplectic parabolic bundles $(E_*,\psi)$ such that $H^0(\operatorname{PEnd}_{\mathrm{Sp}}(E_*)(x))=0$. For every $1\leq i < 2m$ and every $(E_*,\psi) \in \mathcal{Y} = \mathcal{M}_{\textnormal{Sp}}(2m,\alpha,L) \setminus \mathcal{U}$, consider the family of quasi-parabolic bundles obtained by replacing the $i$-th step of the flag of $\left.E\right|_x$ to all allowable (i.e. the modified parabolic bundle with the given symplectic and Higgs structures is an element of the moduli space) subspaces $E'_{x,i}$ such that
	\begin{equation}\label{filtration}
	E_{x,i+1} \subsetneq E'_{x,i} \subsetneq E_{x,i-1}
	\end{equation}
	Let us consider the union of all stable points in such families. Since the codimension of $\mathcal{Y}$ in $\mathcal{M}_{\textnormal{Sp}}(2m,\alpha,L)$ is $\geq 3$ and the families above are at most of dimension $1$, the codimension of the union of all such families must be positive. Therefore, there exist an open subset $\mathcal{W} \subset \mathcal{M}_{\textnormal{Sp}}(2m,\alpha,L)$ which consists of points outside the image of all the families. Consider the open set $\mathcal{V}' = \mathcal{U} \cap \mathcal{V} \cap \mathcal{W}$ of $\mathcal{M}_{\mathrm{Sp}}(2m,\alpha,L)$. Now if we repeat the argument given in \cite[Proposition $4.2$]{AG19} (see \cite[Proposition $3$]{R21} for the symplectic case) by replacing the open subset $\mathcal{V}$ with $\mathcal{V}'$, we can conclude that the generic fibers over $\mathcal{D}_x$ contain a complete rational curve. For the convenience of the readers, we only give a sketch of the argument given in \cite[Proposition $4.2$]{AG19} or \cite[Proposition $3$]{R21}. For any element $(E_*,\psi) \in \mathcal{V}'$ and $x \in D$, the evaluation map 
	\[
	\text{ev} : H^0(\textnormal{PEnd}_{\mathrm{Sp}}(E_*)\otimes K(D)) \to \textnormal{PEnd}_{\mathrm{Sp}}(E_*)\otimes \left.K(D)\right|_x
	\]
	is surjective by Serre duality (since $(E_*,\psi) \in \mathcal{V}$). For $1 < i \leq 2m$, let $N_i(E_*) \subset \textnormal{PEnd}_{\mathrm{Sp}}(E_*)\otimes \left.K(D)\right|_x$ consists of matrices with a zero at $(i-1,i)$. For $i=1$, let $N_1(E_*)$ consists of matrices with a zero at $(2m,1)$. Suppose $\tilde{N}_i(E_*) \coloneqq \textnormal{ev}^{-1}(N_i(E_*))$. Let $E_{*_i}$ denote the parabolic bundle obtained from $E_*$ by removing the subspace $E_{x,i}$ of $\left.E\right|_x$. Then 
	\[
	\tilde{N}_i(E_*) = H^0(\textnormal{PEnd}_{\mathrm{Sp}}(E_{*_i}) \otimes K(D)).
	\]
	
	Also, $(E_*,\psi,\Phi) \in h^{-1}(\mathcal{D}_x)$ if and only if $z^2 | \det(\Phi(z))$, where $\Phi$ is the Higgs field and $z$ is the coordinate around $x\in D$.  Therefore, $z^2|\det(\Phi(z))$ if and only if $\text{ev}(\Phi) \in N_i(E_*)$ for some $1< i \leq 2m$. Since $\text{ev}$ is surjective for all $(E_*,\psi)\in \mathcal{V}'$,
	\[
	h^{-1}(\mathcal{D}_x) \cap T_{(E_*,\psi)}^*\mathcal{M}_{\textnormal{Sp}}(2m,\alpha,L) = \bigcup_{i=1}^{2m}\tilde{N}_i(E_*)
	\]
	Let $(E_{*'},\psi)$ be the stable symplectic parabolic bundle with an allowable filtration (as in \ref{filtration})
	\[
	E_{x,i-1} \supsetneq E_{x,i}' \supsetneq E_{x,i+1}
	\] 
	for all $x\in D$ and all $1< i< 2m$. Since $\Phi$ maps $E_{x,i-1}$ to $E_{x,i+1}$, we have $\Phi \in H^0(\textnormal{PEnd}_{\mathrm{Sp}}(E_{*'})\otimes K(D))$ for all such allowable subspaces $E_{x,i}'$. Therefore $E_{*'} \in h_{\mathcal{U}}^{-1}(\mathcal{D}_x)$ for all such $E_{x,i}'$ and they actually lie in the same fiber. The space of possible compatible steps in this filtration is parametrized by $\mathbb{P}^1$, and hence they form a complete rational curve.

\end{proof}

\begin{prop}\label{1}
Let $\mathcal{U} \subset \mathcal{M}_{\mathrm{Sp}}(2m,\alpha,L)$ be an open subset such that the complement has codimension $\geq 2$. The global algebraic functions $\Gamma(T^*\mathcal{U})$ produce a map
\[
\tilde{h} : T^*\mathcal{U}  \longrightarrow \mathrm{Spec}(\Gamma(T^*\mathcal{U}))\cong \mathcal{A} \cong \mathbb{C}^N,
\]
which is the restriction of the Hitchin map to $T^*\mathcal{U}$ upto an automorphism of $\mathbb{C}^N$, where $N=\dim \mathcal{A}$. Moreover, if we consider the standard dilation action of $\mathbb{C}^*$ on the fibers of the cotangent bundle $T^*\mathcal{U}$, then there is a unique $\mathbb{C}^*$-action on $\mathcal{A}$ such that $\tilde{h}$ is $\mathbb{C}^*$-equivariant, i.e. $\tilde{h}(E_*, \lambda\Phi)= \lambda \cdot \tilde{h}(E_*,\Phi)$.
\end{prop}	
\begin{proof}
	Since $\mathcal{U} \subset \mathcal{M}_{\mathrm{Sp}}(2m,\alpha,L)$ is an open subset of codimension at least $2$, the codimension of the open subset $T^*\mathcal{U} \subset T^*\mathcal{M}_{\mathrm{Sp}}(2m,\alpha,L)$ is at least $2$. Therefore, by Hartog's theorem we have
	\[
	\Gamma(T^*\mathcal{U}) = \Gamma(T^*\mathcal{M}_{\mathrm{Sp}}(2m,\alpha,L)).
	\]
	Hence, the statement follows from \cite[Proposition $5$]{R21}.
\end{proof}
The Proposition \ref{1} allows us to recover the Hitchin map up to an automorphism of the base. The $\mathbb{C}^*$-action stratifies the vector space $\mathcal{A}$ in subspaces corresponding to the points which has a rate of decay at least $|\lambda|^i$ for all $i=2,4,\dots, 2m$. In other words, the subspaces $\mathcal{A}_{\geq 2k} = \bigoplus_{i=k}^{m}\mathcal{A}_{2i}$ (where $\mathcal{A}_{2i} = H^0(X,K^{2i}(D^{2i-1}))$) are uniquely determined for each $k=1,\dots, m$. Therefore, the $\mathbb{C}^*$-action allows us to recover the subspace with maximal decay $|\lambda|^{2m}$, which corresponds to the subspace
\[
\mathcal{A}_{2m} = H^0(X, K^{2m}(D^{2m-1})) \subset \mathcal{A}.
\]

	\begin{prop}\cite[Proposition $6$]{R21}\label{5}
	The intersection $\mathcal{C} \coloneqq \mathcal{A}_{2m} \cap \mathcal{D} \subset \mathcal{A}_{2m}$ has $n+1$ irreducible components
	\[
	\mathcal{C} = \mathcal{C}_X \cup \bigcup_{x \in D} \mathcal{C}_x.
	\]
	Moreover, $\mathbb{P}(\mathcal{C}_X ) \subset \mathbb{P}(\mathcal{A}_{2m})$ is the dual variety of $X \subset \mathbb{P}(\mathcal{A}_{2m}^\vee)$ and for each $x \in D$, $\mathbb{P}(\mathcal{C}_x) \subset \mathbb{P}(\mathcal{A}_{2m})$ is the dual variety of $x \xhookrightarrow{} X \subset \mathbb{P}(\mathcal{A}_{2m}^\vee)$ for the embedding given by the linear series $|K^{2m}D^{2m-1}|$.
\end{prop}

The explicit description of the above components $\mathcal{C}_X$ and $\mathcal{C}_x$ are given below (see \cite[Proposition $6$]{R21}):
\begin{align*}
\mathcal{C}_X &= \bigcup_{x\in X} H^0(K^{2m}D^{2m-1}(-2x))\\
\mathcal{C}_x &= H^0(K^{2m}D^{2m-1}(-x)) \hspace{0.3cm} x \in D.
\end{align*}

Therefore, $\mathbb{P}(\mathcal{C}_x) \not \cong \mathbb{P}(\mathcal{C}_X)$ as $\mathbb{P}(\mathcal{C}_x)$ is the dual variety of a point and $\mathbb{P}(\mathcal{C}_X)$ is the dual variety of a compact Riemann surface. Also, $\mathcal{C}_x \subset \mathcal{A}_{2m}$ is an hyperplane for all $x \in D$. So $\mathcal{C}_X \subset \mathcal{C}$ is the only irreducible component which is not an hyperplane in $\mathcal{A}_{2m}$.

\section{Birational Torelli theorems}
Let $X$ be a compact Riemann surface of genus $g$. Let $G$ be a simple simply-connected connected complex affine algebraic group and let $\mathcal{M}_G(X)$ denote the moduli space of semistable principal $G$-bundles over $X$.
\begin{theorem}\label{thm1}
	Let $X$ and $X'$ be two compact Riemann surfaces of genus $\geq 3$. If $\mathcal{M}_G(X)$ is $2$-birational to $\mathcal{M}_{G'}(X')$, then $X$ is isomorphic to $X'$. 
\end{theorem}
\begin{proof}
	By \cite{KN97}, we have $\mathrm{Pic}(\mathcal{M}_G(X)) \cong \mathbb{Z}$. Since $\mathcal{M}_G(X)$ and $\mathcal{M}_{G'}(X')$ are normal projective varieties, by Proposition \ref{keyprop} we have $\mathcal{M}_G(X) \cong \mathcal{M}_{G'}(X')$. Therefore, by \cite{BH12} $X$ is isomorphic to $X'$.
\end{proof}

\begin{theorem}\label{thm2}
	Let $X$ and $X'$ be two compact Riemann surfaces of genus $g \geq 4$ and $g'\geq 4$ respectively with set of marked points $D \subset X$ and $D' \subset X'$. Let $\mathcal{M}_{\textnormal{Sp}}(2m,\alpha,L)$ and $\mathcal{M}_{\textnormal{Sp}}(2m',\alpha',L')$ be the moduli spaces of stable symplectic parabolic bundles over $X$ and $X'$ respectively with degree and rank coprime, small parabolic weights. If $\mathcal{M}_{\textnormal{Sp}}(2m,\alpha,L)$ is $3$-birational to $\mathcal{M}_{\textnormal{Sp}}(2m',\alpha',L')$, then $m=m'$ and $(X,D)$ is isomorphic to $(X',D')$, i.e. there exists an isomorphism $X \cong X'$ sending $D$ to $D'$. 
\end{theorem}

\begin{proof}
	Let $\mathcal{U}$ and $\mathcal{U}'$ be two open subsets of $\mathcal{M}_{\textnormal{Sp}}(2m,\alpha,L)$ and $\mathcal{M}_{\textnormal{Sp}}(2m',\alpha',L')$ respectively of codimension $\geq 3$. Suppose $\tau : \mathcal{U} \xrightarrow{\sim} \mathcal{U}'$ is an isomorphism. Therefore,
	\begin{align}\label{eqn1}
		\begin{split}
	m(2m+1)(g-1)+m^2n &= \dim(\mathcal{U}) \\ &= \dim(\mathcal{U}') \\ &= m'(2m'+1)(g'-1)+m'^2n',
	\end{split}
	\end{align}
where $n=|D|$ and $n'=|D'|$.\\
Also, we have an induced isomorphism $d(\tau^{-1}) : T^*\mathcal{U} \longrightarrow T^*\mathcal{U}'$, which is $\mathbb{C}^*$-equivariant for the standard dilation action. By Proposition \ref{1} there must exist unique $\mathbb{C}^*$-actions on $\mathcal{A}$ and $\mathcal{A}'$ induced from the $\mathbb{C}^*$-action by dilations on the fibers and a $\mathbb{C}^*$-equivariant isomorphism $f: \mathcal{A} \cong \mathrm{Spec}(\Gamma(T^*\mathcal{U})) \longrightarrow \mathrm{Spec}(\Gamma(T^*\mathcal{U}')) \cong \mathcal{A}'$ such that the following diagram commutes
\[
\begin{tikzcd}
	T^*\mathcal{U} \arrow{r}{d(\tau^{-1})} \arrow[swap]{d}{\tilde{h}} & T^*\mathcal{U}' \arrow{d}{\tilde{h}'} \\%
	\mathcal{A} \arrow{r}{f}& \mathcal{A}'
\end{tikzcd}
\]
Since $f$ is $\mathbb{C}^*$-equivariant, it preserves the filtration of subspaces corresponding to the rate of decay, and $f$ takes the subspace of maximum decay of $\mathcal{A}$ to the subspace of maximum decay of $\mathcal{A}'$. Therefore, the number of subspaces in the filtration must be equal and the subspaces of maximum decay must have equal dimension. Hence, $m-1 = m'-1$, i.e. $m=m'$, and $\dim \mathcal{A}_{2m}= \dim \mathcal{A}'_{2m}$. By Riemann-Roch theorem
\[
\dim \mathcal{A}_{2m} = h^0(K_{X}^{2m}(D^{2m-1}))= (4m-1)(g-1)+(2m-1)n.
\]
Therefore,
\begin{align}
	\begin{split}\label{eqn2}
(4m-1)(g-1)+(2m-1)n &= h^0(K_{X}^{2m}(D^{2m-1}))\\ &= h^0(K_{X'}^{2m}((D')^{2m-1}))\\ &= (4m-1)(g'-1)+(2m-1)n'
\end{split}
\end{align}
The equations in \ref{eqn1} and \ref{eqn2} together gives a system of equations
\begin{align*}
(2m^2+m)(g-g')+ m^2(n-n')&=0 \hspace{0.5cm} \mathrm{and}\\ 
(4m-1)(g-g')+(2m-1)(n-n')&=0.
\end{align*}
Since $m>1$, the above system of equations gives $g=g'$ and $n=n'$. 

 The restriction map $f : \mathcal{A}_{2m} \longrightarrow \mathcal{A}'_{2m}$ is $\mathbb{C}^*$-equivariant and homogeneous of degree $2m$, so it is linear. Since $d(\tau^{-1})$ is an isomorphism, it maps the complete rational curves in $T^*\mathcal{U}$ to the complete rational curves in $T^*\mathcal{U}'$. By Proposition \ref{2}, the locus of singular spectral curves is preserved by $f$, i.e. $f$ sends $\mathcal{D} \subset \mathcal{A}$ to $\mathcal{D}' \subset \mathcal{A}'$. Therefore the restriction map sends $\mathcal{C}= \mathcal{D} \cap \mathcal{A}_{2m}$ to $\mathcal{C}' = \mathcal{D}' \cap \mathcal{A}_{2m}'$, i.e. $f(\mathcal{C})=\mathcal{C}'$. This induces an isomorphism $f^\vee : \mathbb{P}(\mathcal{A}_{2m}^\vee) \longrightarrow \mathbb{P}((\mathcal{A}'_{2m})^\vee)$. Since $\mathcal{C}_X \subset \mathcal{C}$ is canonically identified as the only irreducible component which is not an hyperplane, by Proposition \ref{5} $f^\vee$ sends $X$ to $X'$. Moreover, again by Proposition \ref{5} the divisor $D \subset X$ is the dual of the rest of the components $\mathbb{P}(\mathcal{C}_x) \subset \mathbb{P}(\mathcal{C})$. Therefore, $f^\vee$ must send $D$ to $D'$. Hence, an isomorphism $f^\vee : (X,D) \longrightarrow (X',D')$ is obtained.

\end{proof}

\section*{Acknowledgement}
We sincerely thank the anonymous reviewer whose comments and suggestions helped improve and clarify this manuscript.

\end{document}